\documentclass{amsart}
\usepackage{color}

\usepackage[utf8]{inputenc}
\usepackage{amsmath,amssymb,amsbsy,amsfonts,latexsym}
\newtheorem{thm}{Theorem}
\newtheorem{corollary}{Corollary}[section]

\newtheorem{lemma}{Lemma}[section]

\newtheorem{remark}{Remark}[section]

\newcommand{\dd}{\mathrm{d}}

\begin{document}
\title{Ground state of a magnetic nonlinear Choquard equation}

\author{H. Bueno, G. G. Mamani and G. A. Pereira}

\begin{abstract}
	We consider the stationary magnetic nonlinear Choquard equation 
	\[-(\nabla+iA(x))^2u+ V(x)u=\bigg(\frac{1}{|x|^{\alpha}}*F(|u|)\bigg)\frac{f(|u|)}{|u|}{u},\]
	where $A: \mathbb{R}^{N}\rightarrow \mathbb{R}^{N}$ is a vector potential, $V$ is a scalar potential, $f\colon\mathbb{R}\to\mathbb{R}$ and $F$ is the primitive of $f$. Under mild hypotheses, we prove the existence of a ground state solution for this problem. We also prove a simple multiplicity result by applying Ljusternik-Schnirelmann methods.
\end{abstract}

\maketitle

\noindent keyword: Variational methods, magnetic Choquard equation, splitting lemma\\ 

%\date{}

\section{Introduction}

We consider the problem
\begin{equation}\label{P}
-(\nabla+iA(x))^2u+ V(x)u=\bigg(\frac{1}{|x|^{\alpha}}*F(|u|)\bigg)\frac{f(|u|)}{|u|}{u}
\end{equation}
where $\nabla+iA(x)$ is the covariant derivative with respect to the $C^1$ vector potential
$A: \mathbb{R}^{N}\rightarrow \mathbb{R}^{N}$.  (After stating our hypotheses, the form of equation \eqref{P} will be changed to \eqref{Ptilde}). The constant $\alpha$  belongs to the intervals $(0,N)$ and
\[\lim_{|x|\to \infty}A(x)=A_\infty\in \mathbb{R}^N.\] 

The scalar potential $V\colon \mathbb{R}^{N}\rightarrow\mathbb{R}$ is a continuous, bounded function satisfying
\begin{enumerate}
	\item [($V1$)] $\inf_{\mathbb{R}^N}V>0$;
	\item [($V2$)] $V_\infty=\displaystyle\lim_{|y|\to\infty}V(y)$;
	\item [($V3$)] $V(x)\leq V_\infty$ for all $x\in\mathbb{R}^N$. %, $V(x)\neq V_\infty$.
\end{enumerate} 
We also suppose that
\begin{enumerate}
	\item [($AV$)]	$|A(y)|^2+V(y)<|A_\infty|^2+V_\infty$.
\end{enumerate}

The function $F$ is the primitive of the nonlinearity $f\colon \mathbb{R}\to\mathbb{R}$, which is non-negative in $(0,\infty)$ and satisfies, for any $r\in \left(\frac{2N-\alpha}{N},\frac{2N-\alpha}{N-2}\right)$,
\begin{enumerate}
	\item [($f1$)] $\displaystyle\lim_{t\to 0}\frac{f(t)}{t}=0$,
	\item [($f2$)] $\displaystyle\lim_{t\to\infty}\frac{f(t)}{t^{r-1}}=0$,
	\item [($f3$)] $\displaystyle\frac{f(t)}{t}$ is increasing if $t>0$ and decreasing if $t<0$.
\end{enumerate}

For example, if $t\in\mathbb{R}$, the functions $t \ln(1+|t|) $ and   
$|t|^{q_1-2}t+|t|^{q_2-2}t$ (where $2<q_1,q_2<r$) satisfy hypothesis ($f1$), ($f2$) and ($f3$).

We denote
\[\tilde{f}(t)=\left\{\begin{array}{cl}\displaystyle\frac{f(t)}{t},&\text{if }\ t\neq 0,\\
0, &\text{if }\ t=0.\end{array}\right.\]
Our hypotheses imply that $\tilde{f}$ is continuous. Therefore, problem \eqref{P} can be written in the form
\begin{equation}\label{Ptilde}
-(\nabla+iA(x))^2u+ V(x)u=\bigg(\frac{1}{|x|^{\alpha}}*F(|u|)\bigg)\tilde{f}(|u|){u}.
\end{equation}
The composition of $f$ and $F$ with $|u|$ gives a variational structure to the problem, allowing the application of the Mountain Pass Theorem. So, the right-hand side of problem \eqref{Ptilde} generalizes the term 
\begin{equation}\label{Cingolani}\bigg(\frac{1}{|x|^{\alpha}}*|u|^p)\bigg)|u|^{p-2}{u},
\end{equation}
which was studied by Cingolani, Clapp and Secchi in \cite{Cingolani}.  In some particular cases, similar forms of problem \eqref{Ptilde} were studied in \cite{CCS} and \cite{CSS}. 

Our aim in this paper is to prove the existence of a ground state solution for problem \eqref{Ptilde}. This is accomplished by showing that the mountain pass geometry is satisfied and then considering the asymptotic form of problem \eqref{Ptilde} and applying Struwe's splitting lemma.

The main part of the interesting paper by Cingolani, Clapp and Secchi \cite{Cingolani}  is devoted to the existence of multiple solutions of equation \eqref{Ptilde} - with \eqref{Cingolani} as the right-hand side - under the action  of a closed subgroup $G$ of the orthogonal group $O(N)$ of linear isometries of $\mathbb{R}^N$ if $A(gx)=gA(x)$ and $V(gx)=V(x)$ for all $g\in G$ and $x\in\mathbb{R}^N$. The authors look for solutions satisfying 
\[u(gx) =\tau(g)u(x),\quad\text{for all }\ g\in G\ \ \text{and }\ x\in\mathbb{R}^N,\]
where $\tau\colon G\to S^1$ is a given continuous group homomorphism into the unit complex numbers $S^1$. In this paper we also address the multiplicity of solutions in a particular case of that treated in \cite{Cingolani}.

We define 
\[\nabla_A u=\nabla u+iA(x)u\]
and consider the space
\[H^1_{A,V}(\mathbb{R}^N,\mathbb{C})=\left\{u\in L^2(\mathbb{R}^N,\mathbb{C})\,:\, \nabla_A u\in L^2(\mathbb{R}^N,\mathbb{C})\right\}\]
endowed with scalar product
\[\langle u,v\rangle_{A,V}=\mathfrak{Re}\int_{\mathbb{R}^N}\left(\nabla_A u\cdot \overline{\nabla_A v}+V(x)u\bar v\right)\]
and, therefore
\[\|u\|^2_{A,V}=\int_{\mathbb{R}^N}|\nabla_A u|^2+V|u|^2.\]

Observe that the norm generated by this scalar product is equivalent to the norm obtained by considering $V\equiv 1$, see \cite[Definition 7.20]{LiebLoss}.

If $u\in H^1_{A,V}(\mathbb{R}^N,\mathbb{C})$, then $|u|\in H^1(\mathbb{R}^N)$ and the \emph{diamagnetic inequality} is valid (see \cite[Theorem 7.21]{LiebLoss},\cite{Cingolani})
\[|\nabla |u|(x)|\leq |\nabla u(x)+iA(x)u(x)|,\ \ \textrm{a.e. } x\in\mathbb{R}^N.\]

As a consequence of the diamagnetic inequality, we have the continuous immersion
\begin{equation}\label{immersion}H^1_{A,V}(\mathbb{R}^N,\mathbb{C})\hookrightarrow L^q(\mathbb{R}^N,\mathbb{C})\end{equation}
for any $q\in [2,\frac{2N}{N-2}]$. We denote $2^*=\frac{2N}{N-2}$.

It is well-known that $C^\infty_c(\mathbb{R}^N,\mathbb{C})$ is dense in $H^1_{A,V}(\mathbb{R}^N,\mathbb{C})$, see \cite[Theorem 7.22]{LiebLoss}.
\begin{remark}\label{obs1} 
	It follows from $(f1)$-$(f2)$  that, for any fixed $\xi>0$, there exists a constant $C_\xi$ such that
	\begin{equation}\label{boundf}|f(t)|\leq\xi t+C_\xi t^{r-1},\quad\forall\ t\geq 0.\end{equation}
	Similarly,  there exists $D_\xi>0$ such that
	\[|F(t)| \leq \xi t^2 + D_\xi t^r,\quad \forall\ t \geq 0.\]

	Furthermore, $(f3)$ implies that $f$ satisfies the Ambrosetti-Rabinowitz inequality
	\begin{equation}\label{AR}2F(t)<f(t)t,\quad \forall\ t>0.\end{equation}
	Observe that the function $f(t)=t\ln(1+|t|)$ satisfies the last inequality, but does not satisfy $\theta F(t)\leq tf(t)$ for any $\theta>2$.
\end{remark}

We state our results:
\begin{thm}\label{t1}
	Suppose that $\alpha\in (0,N)$ and that conditions \textup{($V1$)-($V3$)}, \textup{($AV$)} and \textup{($f1$)-($f3$)} are valid. Then, problem \eqref{P} has a ground state solution.  
\end{thm}

In order to obtain our multiplicity result, we define the space
\[H^1_A(\mathbb{R}^N,\mathbb{C})^\tau=\left\{u\in H^1_A(\mathbb{R}^N,\mathbb{C})\,:\, u(gx)=\tau(g)u(x),\ \forall\ g\in G,\ \forall\ x\in \mathbb{R}^N\right\}\]
and suppose that the closed subgroup $G\subset O(N)$ satisfies the decomposition
\begin{equation}\label{decLions}G=O(N_1)\times O(N_2)\times\cdots\times O(N_k),\end{equation}
where $\sum_{j=1}^k N_j=N$, $N_j\geq 2$ for all $j\in\{1,\ldots,k\}$. Then we have
\begin{thm}\label{t2}Let $G$ be a closed subgroup of $O(N)$ satisfying the decomposition \eqref{decLions}. Assume that $A(gx)=gA(x)$ and $V(gx)=V(x)$ for all $g\in G$ and $x\in\mathbb{R}^N$. Then problem \eqref{Ptilde} has a sequence $(u_n)\subset H^1_A(\mathbb{R}^N,\mathbb{C})^\tau$ such that $\displaystyle\lim_{n\to\infty}\|u_n\|^2_{A,V}=\infty$.
\end{thm}
The paper is organized as follows: Section \ref{mpg} shows the mountain pass geometry and some basic results concerning  the right-hand side of equation \eqref{Ptilde}. Theorem \ref{t1} is proved in Section \ref{GS} and our multiplicity result in Section \ref{multiplicity}. 

\section{Variational Formulation}\label{mpg}

The energy functional associated to problem \eqref{P} is given by
\begin{equation}
\label{energy}J_{A,V}(u)=\frac{1}{2}\Vert u\Vert^2_{A,V}-D(u),
\end{equation}
where $$D(u)=\frac{1}{2}\int_{\mathbb{R}^N}\bigg(\frac{1}{|x|^{\alpha}}*F(|u|)\bigg)F(|u|).$$

The energy functional is well-defined as a consequence of the Hardy-Little\-wood-Sobolev (see \cite[Theorem 4.3]{LiebLoss}, since
\begin{align}\label{estconv}\left|\int_{\mathbb{R}^N}\left(\frac{1}{|x|^\alpha}*F(|u|)\right)F(|u|)\right|\leq  C\left(\|u\|^{4}+\|u\|^{2r}\right).\end{align}

\begin{remark}
	\label{obs3} Let us consider the case $F(t)=|t|^r$. By applying the Hardy-Littlewood-Sobolev inequality we have that
	\[\int_{\mathbb{R}^{N}}\left(\frac{1}{|x|^\alpha}*F(|u|)\right)F(|u|)\]
	is well-defined if $F(|u|)\in L^p(\mathbb{R}^{N})$ for $p>1$ defined by
	\[\frac{2}{p}+\frac{\alpha}{N}=2\quad\Rightarrow\quad\frac{1}{p}=\frac{1}{2}\left(2-\frac{\alpha}{N}\right).\]
	Consequently, in order to apply the immersion \eqref{immersion}, we must have
	\begin{align*}pr\in [2,2^*]	&\Rightarrow\quad \frac{2N-\alpha}{N}\leq r\leq \frac{N}{N-2}\left(2-\frac{\alpha}{N}\right)=\frac{2N-\alpha}{N-2}.
	\end{align*}
	This condition (taking the open interval satisfied by $r$) justifies hypothesis ($f2$).
\end{remark}

Since the derivative of the energy functional $J_{A,V}(u)$ is given by
\begin{align*}
J'_{A,V}(u)\cdot \psi&=\langle u,\psi\rangle_{A,V}-D'(u)\cdot \psi\\
&=\langle u,\psi\rangle_{A,V}-\mathfrak{Re}\int_{\mathbb{R}^N}\bigg(\frac{1}{|x|^{\alpha}}*F(|u|)\bigg)\tilde{f}(|u|)u\bar{\psi}, 
\end{align*}
we see that critical points of $J'_{A,V}(u)$ are weak solutions of \eqref{Ptilde}.  Note that, if $\psi=u$ we obtain
\begin{equation}
J'_{A,V}(u)\cdot u:=\Vert u\Vert_{A,V}^{2}-\int_{\mathbb{R}^N}\bigg(\frac{1}{|x|^{\alpha}}*F(|u|)\bigg)f(|u|)|u|. 
\end{equation}

\begin{lemma}\label{gpm}
	The functional $J_{A,V}$ satisfies the Mountain Pass geometry. Precisely,
	\begin{enumerate}
		\item [$(i)$] there exist $\rho,\delta>0$ such that $J_{A,V}\big|_S\geq \delta>0$ for any $u\in S$, where
		\[S=\{u\in H^1_{A,V}(\mathbb{R}^{N},\mathbb{C})\,:\, \|u\|_{A,V}=\rho\};\]
		\item [$(ii)$] for any $u_0\in H^1_{A,V}(\mathbb{R}^{N},\mathbb{C})\setminus\{0\}$ there exists $\tau\in (0,\infty)$ such that $\|\tau u_0\|>\rho$ e  $J_{A,V}(\tau u_0) <0$.
	\end{enumerate}
\end{lemma}
\begin{proof} Inequality \eqref{estconv} yields
	\[J_{A,V}(u)\geq \frac{1}{2}\|u\|^2_{A,V}-C\left(\|u\|^{4}_{A,V}+\|u\|^{2r}_{A,V}\right),\]
	thus implying ($i$) when we take $\|u\|_{A,V}=\rho>0$ small enough.
	
	In order to prove ($ii$), fix $u_0\in H^1_{A,V}(\mathbb{R}^{N},\mathbb{C})\setminus\{0\}$ and consider the function $g_{u_0}\colon(0,\infty)\to\mathbb{R}$ given by
	\[g_{u_0}(t)=D\left(\frac{tu_0}{\|u_0\|_{A,V}}\right)=\frac{1}{2}\int_{\mathbb{R}^{N}}\left[\frac{1}{|x|^\alpha}*F\left(\frac{t|u_0|}{\|u_0\|_{A,V}}\right)\right]F\left(\frac{t|u_0|}{\|u_0\|_{A,V}}\right).\]
	We have	
	\begin{align*}
	g'_{u_0}(t)%&=H'\left(\frac{tu_0}{\|u_0\|}\right)\cdot \frac{u_0}{\|u_0\|}\\
	&=\int_{\mathbb{R}^{N}}\left[\frac{1}{|x|^\alpha}*F\left(\frac{t|u_0|}{\|u_0\|_{A,V}}\right)\right]f\left(\frac{t|u_0|}{\|u_0\|_{A,V}}\right)\frac{|u_0|}{\|u_0\|_{A,V}}\\
	&=\frac{4}{t}\int_{\mathbb{R}^{N}}\frac{1}{2}\left[\frac{1}{|x|^\alpha}*F\left(\frac{t|u_0|}{\|u_0\|_{A,V}}\right)\right]\frac{1}{2}f\left(\frac{t|u_0|}{\|u_0\|_{A,V}}\right)\frac{t|u_0|}{\|u_0\|_{A,V}}\\
	&\geq \frac{4}{t}g_{u_0}(t)\end{align*}
	as a consequence of the Ambrosetti-Rabinowitz condition \eqref{AR}.
	Observe that $g'_{u_0}(t)>0$ for $t>0$.
	
	Thus, 
	\begin{align*}
	\ln g_{u_0}(t)\Big|_1^{\tau\|u_0\|_{A,V}}\geq 4\ln t\Big|_1^{\tau\|u_0\|_{A,V}}\quad\Rightarrow\quad \frac{g_{u_0}(\tau\|u_0\|_{A,V})}{g_{u_0}(1)}\geq \left(\tau\|u_0\|_{A,V}\right)^{4},
	\end{align*}
	proving that
	\begin{align}\label{H}
	D(\tau u_0)=g_{u_0}(\tau\|u_0\|_{A,V})\geq M\left(\tau\|u_0\|_{A,V}\right)^{4}\end{align}
	for a constant $M>0$. So,
	\[J_{A,V}(\tau u_0)=\frac{\tau^2}{2}\|u_0\|^2_{A,V}-D\left(\tau u_0\right)\leq C_1\tau^2-C_2\tau^{4}\]
	yields that $J_{A,V}(\tau u_0)<0$ when $\tau$ is large enough.%, for any $u_0\in  H^1_{A,V}(\mathbb{R}^{N},\mathbb{C})\setminus\{0\}$.
	$\hfill\Box$\end{proof}\vspace*{.4cm}

The mountain pass theorem without the PS condition (see \cite[Teorema. 1.15]{Willem}) yields a Palais-Smale sequence $(u_n)\subset H^1_{A,V}(\mathbb{R}^{N},\mathbb{C})$ such that
\[J'_{A,V}(u_n)\to 0\qquad\textrm{and}\qquad J_{A,V}(u_n)\to c,\]
where
\[c=\inf_{\alpha\in \Gamma}\max_{t\in [0,1]}J_{A,V}(\alpha(t)),\]
and $\Gamma=\left\{\alpha\in C^1\left([0,1],H^1_{A,V}(\mathbb{R}^{N},\mathbb{C})\right)\,:\,\alpha(0)=0,\,\alpha(1)<0\right\}$.

We now consider the Nehari manifold
\begin{align*}\mathcal{N}_{A,V}&=\left\{u\in H^1_{A,V}(\mathbb{R}^{N},\mathbb{C})\setminus\{0\}\,:\,J'_{A,V}(u)\cdot u=0\right\}\\
&=\left\{u\in H^1_{A,V}(\mathbb{R}^{N},\mathbb{C})\setminus\{0\}\,:\,\|u\|^2_{A,V}=\int_{\mathbb{R}^{N}}\left(\frac{1}{|x|^\alpha}*F(|u|)\right)f(|u|)|u|\right\}.
\end{align*}
It is not difficult to see that $\mathcal{N}_{A,V}$ is a manifold in $H^1_{A,V}(\mathbb{R}^{N},\mathbb{C})\setminus\{0\}$. The next result shows that $\mathcal{N}_{A,V}$ is a closed manifold in $H^1_{A,V}(\mathbb{R}^{N},\mathbb{C})$.

\begin{lemma}\label{lN}
	There exists $\beta>0$ such that $\|u\|_{A,V}\geq \beta$ for all $u\in \mathcal{N}_{A,V}$.
\end{lemma}

Another characterization of $c$ in terms of the Nehari manifold is now standard: for $u\neq 0$, consider the function $\Phi(t)=(1/2)\|tu\|^2_{A,V} -D(tu)$, preserving the notation of Lemma \ref{gpm}. The proof of Lemma \ref{gpm} assures that $\Phi(tu)>0$ for $t$ small enough, $\Phi(tu)<0$ for $t$ large enough and $g'_u(t)>0$ if $t>0$. Therefore, $\max_{t\geq 0}\Phi(t)$ is achieved at a unique $t_u=t(u)>0$ and $\Phi'(tu)>0$ for $t<t_u$ and $\Phi'(tu)<0$ for $t>t_u$. Furthermore, $\Psi'(t_uu)=0$ implies that $t_uu\in \mathcal{N}_{A,V}$.

The map $u\mapsto t_u$ ($u\neq 0$) is continuous and $c=c^*$, where
\[c^*=\inf_{u\in H^1_{A,V}(\mathbb{R}^{N},\mathbb{C})\setminus\{0\}}\max_{t\geq 0} J_{A,V}(tu).\] For details, see \cite[Section 3]{Rabinowitz} or \cite{Felmer}.

Standard arguments prove the next affirmative:
\begin{lemma}\label{bounded}
	Let $(u_n)\subset H^1_{A,V}(\mathbb{R}^{N},\mathbb{C})$ be a sequence such that $J_{A,V}(u_n)\to c$ and $J'_{A,V}(u_n)\to 0$, where
	\[c=\inf_{u\in H^1_{A,V}(\mathbb{R}^{N},\mathbb{C})\setminus\{0\}}\max_{t\geq 0} J_{A,V}(tu).\]
	Then $(u_n)$ is bounded and (for a subsequence) $u_n\rightharpoonup u_0$ in $H^1_{A,V}(\mathbb{R}^{N},\mathbb{C})$.
\end{lemma}

\begin{lemma}\label{lK}
	Let $U\subseteqq \mathbb{R}^{N}$ be any open set. For $1<p<\infty$, let $(f_n)$ be a bounded sequence in $L^p(U,\mathbb{C})$ such that $f_n(x)\to f(x)$ a.e. Then $f_n\rightharpoonup f$.
\end{lemma}

The proof of Lemma \ref{lK} follows by adapting the arguments given for the real case, as in \cite[Lemme 4.8, Chapitre 1]{Kavian}.\vspace*{.2cm}

\begin{lemma}\label{lemmaw}
	Suppose that $u_n\rightharpoonup u_0$ in $H^1_{A,V}(\mathbb{R}^N,\mathbb{C})$ and $u_n(x)\to u_0(x)$ a.e. in $\mathbb{R}^N$. Then
	\begin{equation}\label{F}\frac{1}{|x|^\alpha}*F(|u_n(x)|)\rightharpoonup \frac{1}{|x|^\alpha}*F(|u_0(x)|\quad\text{in }\ L^{2N/\alpha}(\mathbb{R}^N).\end{equation}
\end{lemma}
\begin{proof}
	In this proof we adapt some ideas of \cite{Alves}.
	
	The growth condition implies that $F(|u_n|)$ is bounded in $L^{\frac{2N-\alpha}{N-2}}(\mathbb{R}^N)$. Since we can suppose that $u_n(x)\to u_0(x)$ a.e. in $\mathbb{R}^N$, it follows from the continuity of $F$ that $F(|u_n(x)|)\to F(|u_0(x)|)$. From Lemma \ref{lK} follows \[F(|u_n(x)|)\rightharpoonup F(|u_0(x)|).\] 
	
	As a consequence of the Hardy-Littlewood-Sobolev inequality, we have that
	\[\frac{1}{|x|^\alpha}*w(x)\in L^{2N/\alpha}(\mathbb{R}^N)\]
	for all $w\in L^{\frac{2N-\alpha}{N-2}}(\mathbb{R}^N)$; this is a bounded linear operator from $L^{\frac{2N-\alpha}{N-2}}(\mathbb{R}^N)$ to $L^{2N/\alpha}(\mathbb{R}^N)$. A new application of Lemma \ref{lK} yields \eqref{F}.
	$\hfill\Box$\end{proof}
\begin{corollary}\label{corD}Consider
	$$D(u)=\frac{1}{2}\int_{\mathbb{R}^N}\bigg(\frac{1}{|x|^{\alpha}}*F(|u|)\bigg)F(|u|).$$
	
	If $u_n\rightharpoonup u_0$ in $H^1_{A,V}(\mathbb{R}^N,\mathbb{C})$ and $u_n(x)\to u_0(x)$ a.e. in $\mathbb{R}^N$, then $D(u_n)\to D(u_0)$ and $D(u_n-u_0)\to 0$.
\end{corollary}
\begin{proof}
	\begin{align}\label{convol}
	D(u_n)-D(u_0)&=\int_{\mathbb{R}^{N}}\!\left(\frac{1}{|x|^\alpha}*F(|u_n|)\right)F(|u_n|)-\int_{\mathbb{R}^{N}}\!\left(\frac{1}{|x|^\alpha}*F(|u_0|)\right)F(|u_0|)\notag\\
	&=\int_{\mathbb{R}^{N}}\left(\frac{1}{|x|^\alpha}*F(|u_n|)\right)\left[F(|u_n|)-F(u_0)\right]\notag\\
	&\qquad +\int_{\mathbb{R}^{N}}\left(\frac{1}{|x|^\alpha}*\left[F(|u_n|)-F(|u_0|)\right]\right)F(|u_0|).
	\end{align}It follows from Lemma \ref{lemmaw} that
	\[\left(\frac{1}{|x|^\alpha}*F(|u_n|)\right)\] 
	is bounded. Since $F$ is continuous, we have $F(|u_n(z)|)- F(|u_0(z)|)=0$ a.e. in $\mathbb{R}^N$. So, both integrals in \eqref{convol} go to zero when $n\to\infty$ and we are done.
	$\hfill\Box$\end{proof}

\begin{corollary}\label{lemmaconvderiv}Suppose that $u_n\rightharpoonup u_0$ and consider 
	\begin{align*}D'(u_n)\cdot\psi=\mathfrak{Re}\int_{\mathbb{R}^N} \left[\frac{1}{|x|^\alpha} * F(|u_n|)\right] \tilde{f}(|u_n|)(u_n)\overline{\psi},
	\end{align*}
	for $\psi\in C^\infty_c(\mathbb{R}^{N},\mathbb{C})$. 
	Then $D'(u_n)\cdot\psi\to D'(u_0)\cdot\psi$.
\end{corollary}
\begin{proof}
	It follows from the growth condition on $f$ that $\tilde{f}(|u_n|)$ is bounded in $L^p(\mathbb{R}^N)$. Since $u_n(x)\to u_0(x)$ a.e. in $\mathbb{R}^N$ and $\tilde{f}$ is continuous, by applying Lemma \ref{lK} we conclude that
	\begin{align}\label{tildef}\tilde{f}(|u_n|)u_n\rightharpoonup \tilde{f}(|u_0|)u_0\quad\text{in }\ L^{q}(\mathbb{R}^N,\mathbb{C}).\end{align}
	Thus,
	\[\left|\int_{\mathbb{R}^N} \left[\frac{1}{|x|^\alpha} * F(|u_n|)\right]\tilde{f}(|u_n|) u_n\overline{\psi}-\int_{\mathbb{R}^N} \left[\frac{1}{|x|^\alpha} * F(|u_0|)\right] \tilde{f}(|u_0|)v\overline{\psi}\right|\hspace*{1.5cm}\]
	\begin{align*}
	&\leq \left|\int_{\mathbb{R}^N} \frac{1}{|x|^\alpha} * F(|u_n|)\left(\tilde{f}(|u_n|) u_n- \tilde{f}(|u_0|)u_0\right)\overline{\psi}\right|\\
	&\qquad+ \left|\int_{\mathbb{R}^N} \frac{1}{|x|^\alpha} *\left[F(|u_n|)-F(|u_0|)\right]\tilde{f}(|u_0|)u_0\overline{\psi}\right|.
	\end{align*}
	%	\[\int_{\mathbb{R}^N} \left[\frac{1}{|x|^\alpha} * F(|v_n|)\right] f(|v_n|) \frac{v_n\overline{\psi_n}}{|v_n|},\]
	The claim follows from Lemma \ref{lemmaw} and \eqref{tildef}.
	$\hfill\Box$\end{proof}
\section{Ground state}\label{GS}

In order to consider the general case of the potential $V(y)$, we adapt a well-known result due to M. Struwe:

Let $(u_n)$ be the minimizing sequence given as consequence of Lemma \ref{gpm}, that is, $(u_n)\subset H^1_{A,V}(\mathbb{R}^{N},\mathbb{C})$ such that
\[J'_{A,V}(u_n)\to 0\qquad\textrm{and}\qquad J_{A,V}(u_n)\to c,\]
where
\[c=\inf_{u\in H^1_{A,V}(\mathbb{R}^{N},\mathbb{C})\setminus\{0\}}\max_{t\geq 0} J_{A,V}(tu).\]
We assume that $u_n\rightharpoonup u_0\in H^1_{A,V}(\mathbb{R}^{N},\mathbb{C})$. We define $u^1_n=u_n-u_0$ and consider the limit problem
\begin{equation}\label{Pinfty}
-(\nabla+iA_\infty)^2u+ V_\infty u=\bigg(\frac{1}{|x|^{\alpha}}*F(|u|)\bigg)\frac{f(|u|)}{|u|}u,
\end{equation}
where $A_\infty=\displaystyle\lim_{|x|\to \infty} A(x)$ and $V_\infty=\displaystyle\lim_{|x|\to \infty}V(y)$. The energy functional attached to this problem is, of course,
\[J_\infty(u)=\frac{1}{2}\|u\|^2_{A_\infty,V_\infty}-D(u). \]
\begin{lemma}[Splitting Lemma]\label{Struwe} Let $(u_n)\subset H^1_{A,V}(\mathbb{R}^{N},\mathbb{C})$ be such that
	\[J_{A,V}(u_n)\to c,\qquad J'_{A,V}(u_n)\to 0\]
	and $u_n\rightharpoonup u_0$ weakly on $H^1_{A,V}(\mathbb{R}^{N},\mathbb{C})$. Then $J'_{A,V}(u_0)=0$ and we have \emph{either}
	\begin{enumerate}
		\item [($i$)] $u_n\to u_0$ strongly on $H^1_{A,V}(\mathbb{R}^{N},\mathbb{C})$;
		\item [($ii$)] or there exist $k\in\mathbb{N}$, $(y^j_n)\in\mathbb{R}^N$ such that $|y^j_n|\to\infty$ for $j\in \{1,\ldots,k\}$ and nontrivial solutions $u^1,\ldots,u^k$ of problem \eqref{Pinfty} so that
		\[J_{A,V}(u_n)\to J_{A,V}(u_0)+\sum_{j=1}^k J_\infty(u_j)\]
		and
		\[\left\|u_n-u_0-\sum_{j=1}^ku^j(\cdot-y^j_n)\right\|\to 0.\]
	\end{enumerate}
\end{lemma}
\begin{proof}(Sketch) We simply adapt the arguments presented in in \cite[Lemma 2.3]{FMM} and \cite[Theorem 8.4]{Willem}. Since $D'(u_n)\cdot\phi\to D'(u_0)\cdot\phi$, it follows that $J'_{A,V}(u_0)\cdot \phi=0$.
	
	By setting $u^1_n=u_n-u_0$, we have
	\begin{enumerate}
		\item [($a_1$)] $\|u^1_n\|^2_{A,V}=\|u_n\|^2_{A,V}-\|u_0\|^2_{A,V}+o_n(1)$;
		%$\langle u_n-u_0,u_n-u_0\rangle=\|u_n\|^2+\|u_0\|^2-2\langle u_n,u_0\rangle\to \|u_n\|^2-\|u_0\|^2$ when $n\to\infty$;
		\item [($b_1$)] $J_\infty(u^1_n)\to c-J_{A,V}(u_0)$;
		\item [($c_1$)] $J'_\infty(u^1_n)\to 0$.
	\end{enumerate}
	
	Let us define
	\[\delta:=\limsup_{n\to\infty}\sup_{y\in\mathbb{R}^N}\int_{B_1(y)}|u^1_n|^2\dd x.\]
	If $\delta=0$, it follows that $u^1_n\to 0$ in $L^t(\mathbb{R}^N)$ for all $t\in (2,2^*)$. It follows that $u^1_n\to 0$ in $H^1_{A,V}(\mathbb{R}^N,\mathbb{C})$, since $J'_\infty(v^1_n)\to 0$. In this case, the proof of Lemma \ref{Struwe} is complete.
	
	So, let us suppose that $\delta>0$. Then, we obtain a sequence $(y^1_n)\subset\mathbb{R}^N$ such that
	\[\int_{B_1(y_n)}|u^1_n|^2\dd x\geq\frac{\delta}{2}.\]
	
	By setting $v^1_n=u^1_n(\cdot+y^1_n)$, we obtain a new bounded sequence $(v^1_n)$. Therefore, we assume that $v^1_n\rightharpoonup v_1$ in $H^1_{A,V}(\mathbb{R}^N,\mathbb{C})$ and $v^1_n\to v$ a.e. in $\mathbb{R}^N$. Since
	\[\int_{B_1(0)}|v^1_n|^2\dd x>\frac{\delta}{2},\]
	we conclude that $u^1\neq 0$ as consequence of Sobolev's immersion. We also conclude that $(y_n)$ is unbounded, since $u^1_n\rightharpoonup 0$ in $H^1_{A,V}(\mathbb{R}^{N},\mathbb{C})$. Therefore, we may assume that $|y^1_n|\to \infty$. Then, it is easy to see that $J'_\infty(u^1)=0$.
	
	We now define $u^2_n=u^1_n-u^1(\cdot-y_n)$. We then have
	\begin{enumerate}
		\item [($a_2$)] $\|u^2_n\|^2_{A,V}=\|u_n\|^2_{A,V}-\|u_0\|^2_{A,V}-\|u^1\|^2_{A,V}+o_n(1)$;
		%$\langle u_n-u_0,u_n-u_0\rangle=\|u_n\|^2+\|u_0\|^2-2\langle u_n,u_0\rangle\to \|u_n\|^2-\|u_0\|^2$ when $n\to\infty$;
		\item [($b_2$)] $J_\infty(u^2_n)\to c-J_{A,V}(u_0)-J_\infty(u_1)$;
		\item [($c_2$)] $J'_\infty(u^2_n)\to 0$.
	\end{enumerate}
	
	Proceeding by iteration, we observe that, if $u$ is a nontrivial critical point of $J_\infty$ and $\bar u$ a ground state of problem \eqref{Pinfty}, then the Ambrosetti-Rabinowitz condition implies that 
	\[J_\infty(u)\geq J_\infty(\bar u)=\int_{\mathbb{R}^{N}}\left(\frac{1}{2}f(|\bar u|)|\bar u|-F(|\bar u|)\right)=:\beta>0.\]
	
	Therefore, it follows from ($b_2$) that the iteration process must end at some index $k\in\mathbb{N}$. 
	$\hfill\Box$\end{proof}\vspace*{.1cm}

\begin{remark}\label{remasymp}	Observe that, in particular, the proof shows that the sequence $u^k_n$ converges to $\bar u$ and we have a solution of problem \eqref{Pinfty}.
\end{remark}
The next result also follows \cite[Corollary 2.3]{FMM}, see also \cite{BMP}. We present the proof for the convenience of the reader.
\begin{lemma}\label{PS}
	The functional $J_{A,V}$ satisfies $(PS)_c$ for any $0\leq c<c_\infty$.
\end{lemma}
\begin{proof}Let us suppose that $(u_n)$ satisfies
	\[J_{A,V}(u_n)\to c<c_\infty\qquad\text{and}\qquad J'_{A,V}(u_n)\to 0.\]
	
	According to Lemma \ref{bounded}, we can suppose that the sequence $(u_n)$ is bounded. Therefore, for a subsequence, we have $u_n\rightharpoonup u_0$ in $H^1_{A,V}(\mathbb{R}^{N},\mathbb{C})$. It follows from the Splitting Lemma \ref{Struwe} that $J'_{A,V}(u_0)=0$. Since
	\begin{align*}
	J'_{A,V}(u_0)\cdot u_0&=\Vert u_0\Vert_{A,V}^{2}-\int_{\mathbb{R}^N}\bigg(\frac{1}{|x|^{\alpha}}*F(|u_0|)\bigg)f(|u_0|)|u_0|
	\end{align*}
	we conclude that
	\begin{align}\label{Iu0}
	J_{A,V}(u_0)&=\frac{1}{2}\int_{\mathbb{R}^{N}}\left[\frac{1}{|x|^\alpha}*F(|u_0|)\right]\left(f(|u_0|)|u_0|-2F(|u_0|)\right)\nonumber\\
	&\quad+\frac{1}{2}\int_{\mathbb{R}^{N}}\left[\frac{1}{|x|^\alpha}*F(|u_0|)\right]F(|u_0|)\nonumber\\
	&>\frac{1}{2}\int_{\mathbb{R}^{N}}\left[\frac{1}{|x|^\alpha}*F(|u_0|)\right]\left(f(|u_0|)|u_0|-2F(|u_0|)\right)>0
	\end{align}
	as a consequence of the Ambrosetti-Rabinowitz condition.
	
	If $u_n\not\to u_0$ in $H^1_{A,V}(\mathbb{R}^{N},\mathbb{C})$, by applying again the Splitting Lemma we guarantee the existence of $k\in\mathbb{N}$ and nontrivial solutions $u^1,\ldots,u^k$ of problem \eqref{Pinfty} satisfying
	\[\lim_{n\to\infty}J_{A,V}(u_n)=c=J_{A,V}(u_0)+\sum_{j=1}^kJ_\infty(u^j)\geq kc_\infty\geq c_\infty\]
	contradicting our hypothesis. We are done.
	$\hfill\Box$\end{proof}\vspace*{.2cm}

We prove the next result by adapting the proof given in Furtado, Maia e Medeiros \cite[Proposition 3.1]{FMM}, see also \cite{BMP}:
\begin{lemma}\label{ccinfty}Suppose that $V(y)$ satisfies $(V_3)$. Then
	\[0<c<c_\infty,\]
	where $c$ is characterized in Lemma \ref{bounded}.
\end{lemma}
\begin{proof}Let $\bar u$ be the weak solution of \eqref{Pinfty} obtained in the proof of the Splitting Lemma (see Remark \ref{remasymp}) and $t_{\bar u}>0$ the unique number such that $t_{\bar u}\bar u\in \mathcal{N}_{A,V}$. We claim that $t_{\bar u}<1$. Indeed, it follows from the condition ($AV$) that
	\[\int_{\mathbb{R}^{N}}\left[\frac{1}{|x|^\alpha}*F(|t_{\bar u}\bar u|)\right]f(|t_{\bar u}\bar u|)|t_{\bar u}\bar u|=t^2_{\bar u}\|\bar u\|_{A,V}\hspace*{4.8cm}\]
	\begin{align*}
	&<t^2_{\bar u}\|\bar u\|_{A_\infty,V_\infty}\\
	&=t^2_{\bar u}\int_{\mathbb{R}^{N}}\left[\frac{1}{|x|^\alpha}*F(|\bar u|)\right]f(|\bar u|)|\bar u|\\
	&=t^2_{\bar u}\left(\int_{\mathbb{R}^{N}}\left[\frac{1}{|x|^\alpha}*F(|\bar u|)\right]f(|\bar u|)|\bar u|+\int_{\mathbb{R}^{N}}\left[\frac{1}{|x|^\alpha}*F(|t_{\bar u}\bar u|)\right]f(|\bar u|)|\bar u|\right.\\
	&\qquad\quad\left.-\int_{\mathbb{R}^{N}}\left[\frac{1}{|x|^\alpha}*F(|t_{\bar u}\bar u|)\right]f(|\bar u|)|\bar u|\right)
	\end{align*}
	thus yielding
	\begin{align*}
	0&>\int_{\mathbb{R}^{N}}\left[\frac{1}{|x|^\alpha}*F(|t_{\bar u}\bar u)|\right]\left(\tilde{f}(|t_{\bar u}\bar u|)-\tilde{f}(|\bar u|)\right)\\
	&\qquad +t^2_{\bar u}\int_{\mathbb{R}^{N}}\left[\frac{1}{|x|^\alpha}*\left(F(|t_{\bar u}\bar u|)-F(|\bar u|)\right)\right]f(|\bar u|)|\bar u|.
	\end{align*}	
	
	If $t_{\bar u}\geq 1$, since $\tilde{f}$ is increasing, the first integral is non-negative and the second as well, since $F$ is also increasing. We conclude that $t_{\bar u}<1$.
	
	Lemma \ref{bounded} and its previous comments show that
	\begin{align*}c\leq \max_{t\geq 0}J_{A,V}(t\bar u)&=J_{A,V}(t_{\bar u}\bar u)\\
	&=\int_{\mathbb{R}^{N}}\left[\frac{1}{|x|^\alpha}*F(|t_{\bar u}\bar u|)\right]\left(\frac{1}{2}f(|t_{\bar u}\bar u|)|t_{\bar u}\bar u|-F(|t_{\bar u}\bar u|)\right).\end{align*}
	Since
	\[g(t)=\int_{\mathbb{R}^{N}}\left[\frac{1}{|x|^\alpha}*F(|t\bar u|)\right]\left(\frac{1}{2}f(|t\bar u|)|t\bar u|-F(|t\bar u|)\right) \]
	is a strictly increasing function, we conclude that
	\[c=g(t_{\bar u})<g(1)=\int_{\mathbb{R}^{N}}\left[\frac{1}{|x|^\alpha}*F(|\bar u|)\right]\left(\frac{1}{2}f(|\bar u|)|\bar u|-F(|\bar u|)\right)=c_\infty,\]
	proving our result. $\hfill\Box$\end{proof}\vspace*{.4cm}

\noindent\textit{Proof of Theorem \ref{t1}}. Let $(u_n)$ be the minimizing sequence given by Lemma \ref{gpm}. It follows from Lemmas \ref{PS} and \ref{ccinfty} that $u_n$ converges to $u\in H^1_{A,V}(\mathbb{R}^N,\mathbb{C})$ satisfying  $J_{A,V}(u)=c$ and $J'_{A,V}(u)=0$.
$\hfill\Box$

\section{On the multiplicity of solutions}\label{multiplicity}
In order to obtain multiplicity of solutions, we consider in this section a particular case of that considered by Cingolani, Clapp and Secchi in \cite{Cingolani}. We think that the direct proof we present is interesting.

So, let $G$ be a closed subgroup of $O(n)$, the group of orthogonal transformations in $\mathbb{R}^N$. As in \cite{Cingolani}, we suppose that $A(gx)=gA(x)$ and $V(gx)=V(x)$ for all $g\in G$ and $x\in\mathbb{R}^N$ and take a continuous group homomorphism $\tau\colon G\to S^1$ into the unit complex numbers $S^1$. 

We consider the space
\[H^1_A(\mathbb{R}^N,\mathbb{C})^\tau=\left\{u\in H^1_A(\mathbb{R}^N,\mathbb{C})\,:\, u(gx)=\tau(g)u(x),\ \forall\ g\in G,\ \forall\ x\in \mathbb{R}^N\right\}.\]
We apply the following compactness result due to P.L. Lions:
\begin{lemma}[Lions]\label{lLions}Let $G$ be a closed subgroup of $O(N)$ and denote 
	\[H^1_G=\left\{u\in H^1(\mathbb{R}^N)\,:\, gu=u,\ \forall\ g\in G\right\}.\]
	
	Suppose that $\sum_{j=1}^k N_j=N$, $N_j\geq 2$ for all $j\in\{1,\ldots,k\}$, and
	\[G=O(N_1)\times O(N_2)\times\cdots\times O(N_k).\]
	
	Then, the immersion $H^1_G(\mathbb{R}^N)\subset L^p(\mathbb{R}^N)$ is compact for $2<p<2^*$.
\end{lemma}

Observe that, if $u\in  H^1_A(\mathbb{R}^N,\mathbb{C})^\tau$, then $|u|\in H^1_G(\mathbb{R}^N)$.\vspace*{.4cm}

\noindent\textit{Proof of Theorem \ref{t2}}.
It follows from applying Theorem 10.10 from Ambrosetti e Malchiodi \cite{AM} to the Nehari manifold $M=\mathcal{N}_{A,V}$. 
$\hfill\Box$

\end{document}